\documentclass[12pt]{article}
\usepackage[margin=1in]{geometry}
\usepackage{amsmath,amsthm,amssymb}
\usepackage{mathrsfs}
\usepackage{amsmath}
\usepackage{amssymb}
\usepackage{epstopdf}
\usepackage{graphicx}
\usepackage{subfig}
\usepackage[utf8]{inputenc}
\usepackage{pgf,tikz}
\usetikzlibrary{arrows}
\usepackage{tikz-cd}
\usepackage{amsfonts}
\usepackage[unicode]{hyperref}
\hypersetup{ unicode = true }
\usepackage[utf8]{inputenc}
\usepackage{fancyhdr}
\usepackage{lastpage}
\usepackage[font=small,labelfont=bf]{caption}
\usepackage{subfig}

\usepackage{etoolbox}
\patchcmd{\thebibliography}{\section*{\refname}}{}{}{}

\usetikzlibrary{arrows}
\tikzset{commutative diagrams/.cd,arrow style=tikz,diagrams={>=latex'}}

\usetikzlibrary{positioning}

\makeatletter
\newcommand*{\rom}[1]{\expandafter\@slowromancap\romannumeral #1@}
\makeatother

\theoremstyle{plain}
\newtheorem{Th}{Theorem}[section]
\newtheorem{Lem}[Th]{Lemma}

\newtheorem{Cor}[Th]{Corollary}

\newtheorem{Claim}[Th]{Claim}
\theoremstyle{definition}

\newtheorem{Conj}[Th]{Conjecture}

\newtheorem{Obs}[Th]{Observation}

\theoremstyle{plain}
\theoremstyle{remark}

\numberwithin{equation}{section}
\date{}

\usepackage{fancyhdr}

\begin{document}

\title{On grids in point-line arrangements in the plane}
\author{Mozhgan Mirzaei\thanks{Department of Mathematics,  University of California at San Diego, La Jolla, CA, 92093 USA.  Supported by NSF grant DMS-1800746. Email:
{\tt momirzae@ucsd.edu}.}\and Andrew Suk\thanks{Department of Mathematics,  University of California at San Diego, La Jolla, CA, 92093 USA. Supported by an NSF CAREER award and an Alfred Sloan Fellowship. Email: {\tt asuk@ucsd.edu}.} }
\maketitle

\begin{abstract} 
The famous Szemer\'{e}di-Trotter theorem states that any arrangement of $n$ points and $n$ lines in the plane determines $O(n^{4/3})$ incidences, and this bound is tight. In this paper, we prove the following Tur\'an-type result for point-line incidence. Let $\mathcal{L}_a$ and $\mathcal{L}_b$ be two sets of $t$ lines in the plane and let $P=\{\ell_a \cap \ell_b :  \ell_a \in \mathcal{L}_a, \ell_b \in \mathcal{L}_b\}$ be the set of intersection points between $\mathcal{L}_a$ and $\mathcal{L}_b$.  We say that $(P, \mathcal{L}_a \cup \mathcal{L}_b)$ forms a \emph{natural $t\times t$ grid} if $|P| =t^2$, and $conv(P)$ does not contain the intersection point of some two lines in $\mathcal{L}_a$ and does not contain the intersection point of some two lines in $\mathcal{L}_b.$ For fixed $t > 1$, we show that any arrangement of $n$ points and $n$ lines in the plane that does not contain a natural $t\times t$ grid determines $O(n^{\frac{4}{3}- \varepsilon})$ incidences, where $\varepsilon  = \varepsilon(t)>0$. We also provide a construction of $n$ points and $n$ lines in the plane that does not contain a natural $2 \times 2$ grid and determines at least $\Omega({n^{1+\frac{1}{14}}})$ incidences.
\end{abstract}

\section{Introduction}
Given a finite set $P$ of points in the plane and a finite set $\mathcal{L}$ of lines in the plane, let $I(P,\mathcal{L}) =\{(p, \ell) \in P\times \mathcal{L}:   p \in \ell \}$ be the set of incidences between $P$ and $\mathcal{L}$.  The \emph{incidence graph} of $(P,\mathcal{L})$ is the bipartite graph $G = (P\cup \mathcal{L}, I)$, with vertex parts $P$ and $\mathcal{L},$ and $E(G)=I(P, \mathcal{L})$. If $|P| = m$ and $|\mathcal{L}| = n$, then the celebrated theorem of Szemer\'edi and Trotter \cite{szemeredi1983extremal} states that 
\begin{equation}\label{szt}
|I(P,\mathcal{L})| \leq O(m^{2/3}n^{2/3} + m + n).
\end{equation} 
Moreover, this bound is tight which can be seen by taking the $\sqrt{m} \times \sqrt{m}$ integer lattice and bundles of parallel "rich" lines (see \cite{pach2011combinatorial}). It is widely believed that the extremal configurations maximizing the number of incidences between $m$ points and $n$ lines in the plane exhibit some kind of lattice structure.  The main goal of this paper is to show that such extremal configurations must contain large \emph{natural grids}.

Let $P$ and $P_0$ (respectively, $\mathcal{L}$ and $\mathcal{L}_0$) be two sets of points (respectively, lines) in the plane. We say that the pairs $(P,\mathcal{L})$ and $(P_0,\mathcal{L}_0)$ are \emph{isomorphic} if their incidence graphs are isomorphic. Solymosi made the following conjecture (see page $291$ in \cite{brass2006research}).

\begin{Conj}For any set of points $P_0$ and for any set of lines $\mathcal{L}_0$ in the plane, the maximum number of incidences
between $n$ points and $n$ lines in the plane containing no
subconfiguration isomorphic to $(P_0,\mathcal{L}_0)$ is $o(n^{\frac{4}{3}}).$
\end{Conj}

In \cite{solymosi2006dense}, Solymosi proved this conjecture in the
special case that $P_0$ is a fixed set of points in the plane, no three of which are on a line, and $\mathcal{L}_0$ consists of all of their connecting lines.  However, it is not known if such configurations satisfy the following stronger conjecture.

\begin{Conj}\label{eps}For any set of points $P_0$ and for any set of lines $\mathcal{L}_0$ in the plane, there is a constant $\varepsilon = \varepsilon(P_0,\mathcal{L}_0)$, such that the maximum number of incidences
between $n$ points and $n$ lines in the plane containing no
subconfiguration isomorphic to $(P_0,\mathcal{L}_0)$ is $O(n^{4/3 - \varepsilon}).$
\end{Conj}

Our first theorem is the following.

\begin{Th}\label{main theorem}
For fixed $t > 1$, let $\mathcal{L}_a$ and $\mathcal{L}_b$ be two sets of $t$ lines in the plane, and let $P_0=\{\ell_a \cap \ell_b: \ell_a \in \mathcal{L}_a, \ell_b \in \mathcal{L}_b\}$ such that $|P_0| = t^2$.  Then there is a constant $c = c(t)$ such that any arrangement of $m$ points and $n$ lines in the plane that does not contain a subconfiguration isomorphic to $(P_0,\mathcal{L}_a \cup \mathcal{L}_b)$ determines at most $c(m^{\frac{2t - 2}{3t - 2}}n^{\frac{2t - 1}{3t - 2}} + m^{1 + \frac{1}{6t-3}} + n)$ incidences.
\end{Th}
See the Figure \ref{unnatural grids}. As an immediate corollary, we prove Conjecture \ref{eps} in the following special case.

\begin{figure}
\centering
\includegraphics[width=0.5\textwidth]{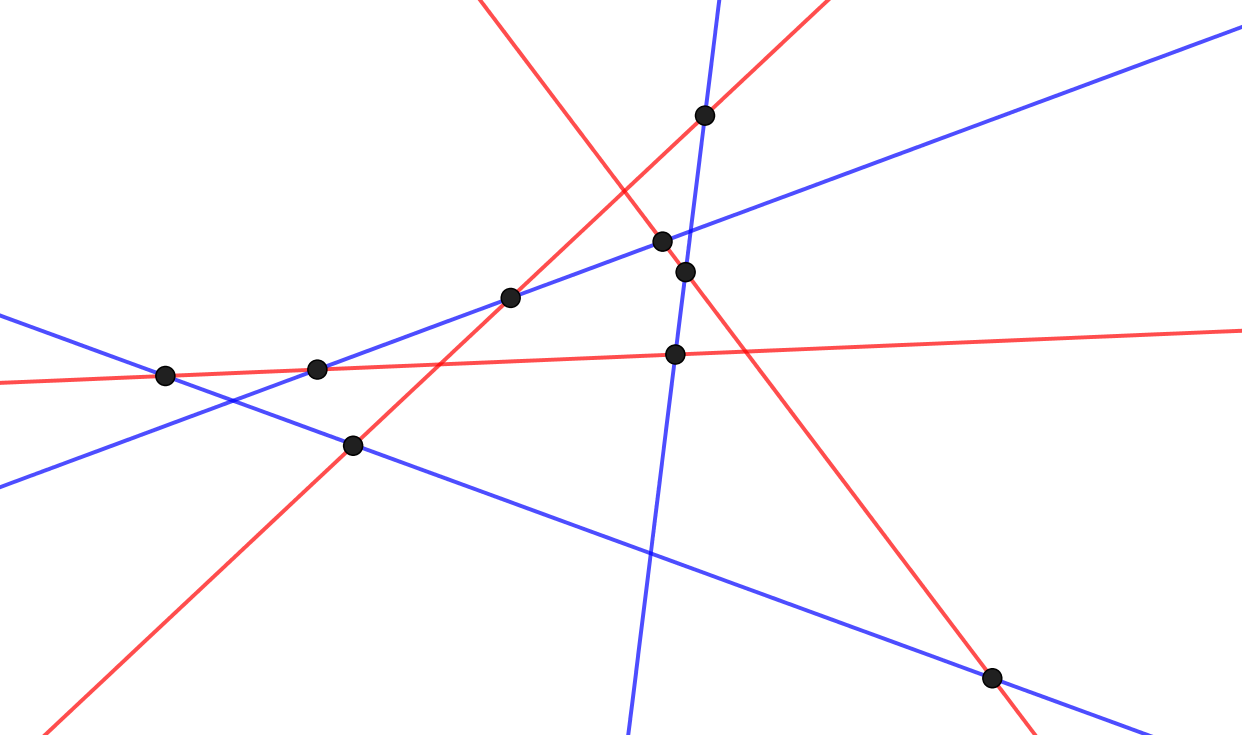}
\caption{An example with $|\mathcal{L}_a|=|\mathcal{L}_b|=3$ and $|P|=9$ in Theorem \ref{main theorem}.} \label{unnatural grids}
\end{figure}

\begin{Cor} \label{first corollary}
For fixed $t > 1$, let $\mathcal{L}_a$ and $\mathcal{L}_b$ be two sets of $t$ lines in the plane, and let $P_0=\{\ell_a \cap \ell_b: \ell_a \in \mathcal{L}_a, \ell_b \in \mathcal{L}_b\}.$  If $|P_0| = t^2$,  then any arrangement of $n$ points and $n$ lines in the plane that does not contain a subconfiguration isomorphic to $(P_0,\mathcal{L}_a \cup \mathcal{L}_b)$ determines at most $O(n^{\frac{4}{3} - \frac{1}{9t - 6}})$ incidences.
\end{Cor}

In the other direction, we prove the following.
\begin{Th}\label{Lower Bound}
Let $\mathcal{L}_a$ and $\mathcal{L}_b$ be two sets of $2$ lines in the plane, and let $P_0=\{\ell_a \cap \ell_b: \ell_1 \in \mathcal{L}_a, \ell_b \in \mathcal{L}_b\}$ such that $|P_0|= 4.$ For $n > 1,$ there exists an arrangement of $n$ points and $n$ lines in the plane that does not contain a subconfiguration isomorphic to $(P_0,\mathcal{L}_a \cup \mathcal{L}_b),$ and determines at least $\Omega({n^{1+\frac{1}{14}}})$ incidences.
\end{Th}

Given two sets $\mathcal{L}_a$ and $\mathcal{L}_b$ of $t$ lines in the plane, and the point set $P_0=\{\ell_a \cap \ell_b: \ell_a \in \mathcal{L}_a, \ell_b \in \mathcal{L}_b\},$ we say that $(P_0,\mathcal{L}_a\cup \mathcal{L}_b)$ forms a \emph{natural $t\times t$ grid} if $|P_0| = t^2,$ and the convex hull of $P_0,$ $conv(P_0),$ does not contain the intersection point of any two lines in $\mathcal{L}_a$ and does not contain the intersection point of any two lines in $\mathcal{L}_b.$ See Figure \ref{natural grids}.

\begin{figure}
\centering
\includegraphics[width=0.45\textwidth]{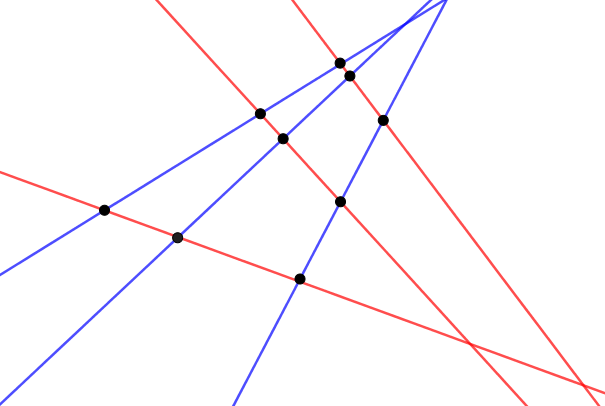}
\caption{An example of a natural $3 \times 3$ grid.}\label{natural grids}
\end{figure}

\begin{Th} \label{Second Corollary}
For fixed $t >1,$ there is a constant $\varepsilon  =\varepsilon(t),$ such that any arrangement of $n$ points and $n$ lines in the plane that does not contain a natural $t \times t$ grid determines at most 
$O(n^{\frac{4}{3} - \varepsilon})$ incidences.
\end{Th}

 Let us remark that $\varepsilon = \Omega(1/t^2)$ in Theorem \ref{Second Corollary}, and can be easily generalized to the off-balanced setting of $m$ points and $n$ lines.

We systemically omit floor and ceiling signs whenever they are not crucial for the sake of clarity of our presentation. All logarithms are assumed to be base $2.$ For $N>0,$ we let $[N]=\{1, \ldots, N\}.$

\section{Proof of Theorem \ref{main theorem}}
In this section we will prove Theorem \ref{main theorem}. We first list several results that we will use. The first lemma is a classic result in graph theory.

\begin{Lem} [K\"{o}vari-S\'{o}s-Tur\'{a}n \cite{kovari1954problem}]\label{Kovari Sos Turan theorem}
Let $G = (V, E)$ be a graph that does not contain a complete bipartite graph $K_{r,s}$  $(1 \leq r \leq s)$ as a subgraph. Then $|E| \leq c_{s}|V|^{2-\frac{1}{r}},$ where $c_{s} > 0$ is constant which only depends on $s.$
\end{Lem}

The next lemma we will use is a partitioning tool in discrete geometry known as \emph{simplicial partitions}. We will use the dual version which requires the following definition. Let $\mathcal{L}$ be a set of lines in the plane. We say that a point $p$ \emph{crosses} $\mathcal{L}$ if it is incident to at least one member of $\mathcal{L},$ but not incident to all members in $\mathcal{L}.$

\begin{Lem}[Matousek \cite{matouvsek1992efficient}]\label{Simplicial Partition Theorem}
Let $\mathcal{L}$ be a set of $n$ lines in the plane and let $r$ be a parameter such that $1 < r < n.$ Then there is a partition on $\mathcal{L}= \mathcal{L}_1 \cup \cdots \cup \mathcal{L}_r$ into $r$ parts, where $\frac{n}{2r} \leq |\mathcal{L}_i| \leq \frac{2n}{r},$ such that any point $p \in \mathbb{R}^2$  crosses at most $O(\sqrt{r})$ parts $\mathcal{L}_i.$
\end{Lem}

\begin{proof}[Proof of Theorem \ref{main theorem}.] Set $t \geq 2.$ Let $P$ be a set of $m$ points in the plane and let $\mathcal{L}$ be a set of $n$ lines in the plane such that $(P, \mathcal{L})$ does not contain a subconfiguration  isomorphic to $(P_0, \mathcal{L}_a\cup \mathcal{L}_b).$

 If $n \geq m^2/100,$ then (\ref{szt}) implies that $|I(P,\mathcal{L})|= O(n)$ and we are done.  Likewise, if $n \leq m^{\frac{t}{2t-1}},$ then (\ref{szt}) implies that $|I(P,\mathcal{L})|= O(m^{1+\frac{1}{6t-3}})$ and we are done. Therefore, let us assume $m^{\frac{t}{2t-1}} < n < m^2/100.$ In what follows, we will show that $|I(P,\mathcal{L})| =O(m^{\frac{2t-2}{3t-2}}n^{\frac{2t-1}{3t-2}}).$
For sake of contradiction, suppose that $I(P, \mathcal{L}) \geq cm^{\frac{2t-2}{3t-2}}n^{\frac{2t-1}{3t-2}},$ where $c$ is a large constant depending on $t$ that will be determined later.

Set $ r= \lceil10 {n^{\frac{4t-2}{3t-2}}}/{m^{\frac{2t}{3t-2}}}\rceil.$
Let us remark that $1 < r < n/10$ since we are assuming $m^{\frac{t}{2t-1}} < n < m^{2}/100.$ We apply Lemma \ref{Simplicial Partition Theorem} with parameter $r$ to $\mathcal{L},$ and obtain the partition $\mathcal{L}= \mathcal{L}_1 \cup \cdots \cup \mathcal{L}_r$ with the properties described above. Note that $|\mathcal{L}_i| > 1.$ Let $G$ be the incidence graph of $(P, \mathcal{L}).$ For $p \in P,$ consider the set of lines in $\mathcal{L}_i.$ If $p$ is incident to exactly one line in $\mathcal{L}_i,$ then delete the corresponding edge in the incidence graph $G.$ After performing this operation between each point $p \in P$ and each part $\mathcal{L}_i,$ by Lemma \ref{Simplicial Partition Theorem}, we have deleted at most $c_1m\sqrt{r}$ 
edges in $G,$ where $c_1$ is an absolute constant. By setting $c$ sufficiently large, we have 

\begin{align*}
c_1m\sqrt{r}=\sqrt{10}c_1 m^{\frac{2t-2}{3t-2}}n^{\frac{2t-1}{3t-2}} < {(c/2)m^{\frac{2t-2}{3t-2}}}{n^{\frac{2t-1}{3t-2}}}.
\end{align*}

\noindent Therefore, there are at least ${(c/2)m^{\frac{2t-2}{3t-2}}}{n^{\frac{2t-1}{3t-2}}}$ edges remaining in $G.$ By the pigeonhole principle, there is a part $\mathcal{L}_i$ such that the number of edges between $P$ and $\mathcal{L}_i$ in $G$ is at least 

\begin{align*}
 \frac{{cm^{\frac{2t-2}{3t-2}}}{n^{\frac{2t-1}{3t-2}}}}{2r}= \frac{cm^{\frac{4t-2}{3t-2}}}{20n^{\frac{2t-1}{3t-2}}} .
\end{align*}

\noindent Hence, every point $p \in P$ has either $0$ or at least $2$ neighbors in $\mathcal{L}_i$ in $G.$ We claim that $(P, \mathcal{L}_i)$ contains a subconfiguration isomorphic to $(P_0, \mathcal{L}_a \cup \mathcal{L}_b).$ To see this, let us construct a graph $H=(\mathcal{L}_i, E)$ as follows. Set $V(H)=\mathcal{L}_i.$ Let $Q=\{q_1, \ldots, q_w\} \subset P$ be the set of points in $P$ that have at least two neighbors in $\mathcal{L}_i$ in the graph $G.$ For $q_j \in Q,$ consider the set of lines $\{\ell_1, \ldots, \ell_s\}$ from $\mathcal{L}_i$ incident to $q_j,$ such that $\{\ell_1, \ldots, \ell_s\}$ appears in clockwise order. Then we define $E_j \subset \binom{\mathcal{L}_i}{2}$ to be a matching on $\{\ell_1, \ldots, \ell_s\},$ where \\

\[
                    E_j= \left\{
                       \begin{array}{ll}
                         \{(\ell_1, \ell_2),(\ell_3, \ell_4), \ldots, (\ell_{s-1}, \ell_s)\}&\text{if } s \text{ is even}.\\ \\
                         
                         \{(\ell_1, \ell_2),(\ell_3, \ell_4), \ldots, (\ell_{s-2}, \ell_{s-1})\}&\text{if }  s \text{ is odd}.
                        \end{array}
                    \right.
                \]  \\

\noindent Set $E(H)=E_1 \cup E_2 \cup \cdots \cup E_w.$ Note that $E_j$ and $E_k$ are disjoint, since no two points are contained in two lines. Since $|E_j| \geq 1,$ we have 

\begin{align*}
|E(H)| \geq \frac{cm^{\frac{4t-2}{3t-2}}}{60n^{\frac{2t-1}{3t-2}}}.
\end{align*}

\noindent Since
\begin{align*}
|V(H)|=|\mathcal{L}_i|\leq \frac{m^{\frac{2t}{3t-2}}}{5n^{\frac{t}{3t-2}}},
\end{align*}

\noindent this implies

\begin{align*}
|E(H)| \geq \frac{c}{60 \cdot 25}(V(H))^{2-\frac{1}{t}}.
\end{align*}

\noindent By setting $c=c(t)$ to be sufficiently large, Lemma \ref{Kovari Sos Turan theorem} implies that $H$ contains a copy of $K_{t,t}.$ Let $\mathcal{L}'_1, \mathcal{L}'_2 \subset \mathcal{L}_i$ correspond  to the vertices of this $K_{t,t}$ in $H,$ and let $P'=\{\ell_1 \cap \ell_2 \in P: \ell_1 \in \mathcal{L}'_1, \ell_2 \in \mathcal{L}'_2 \}.$ We claim that $(P', \mathcal{L}'_1 \cup \mathcal{L}'_2)$ is isomorphic to $(P_0, \mathcal{L}_a \cup \mathcal{L}_b).$ It suffices to show that $|P'|=t^2.$ For the sake of contradiction, suppose $p \in \ell_1 \cap \ell_2 \cap \ell_3,$ where $\ell_1, \ell_2 \in \mathcal{L}'_1$ and $\ell_3 \in \mathcal{L}'_2.$ This would imply $(\ell_1, \ell_3), (\ell_2, \ell_3) \in E_j$ for some $j$ which contradicts the fact that $E_j \subset \binom{\mathcal{L}_i}{2}$ is a matching. Same argument follows if $\ell_1 \in \mathcal{L}'_1$ and $\ell_2, \ell_3 \in \mathcal{L}'_2.$ This completes the proof of Theorem \ref{main theorem}.  \end{proof}

\section{Natural Grids}

\begin{figure}
\centering
\includegraphics[width=0.5\textwidth]{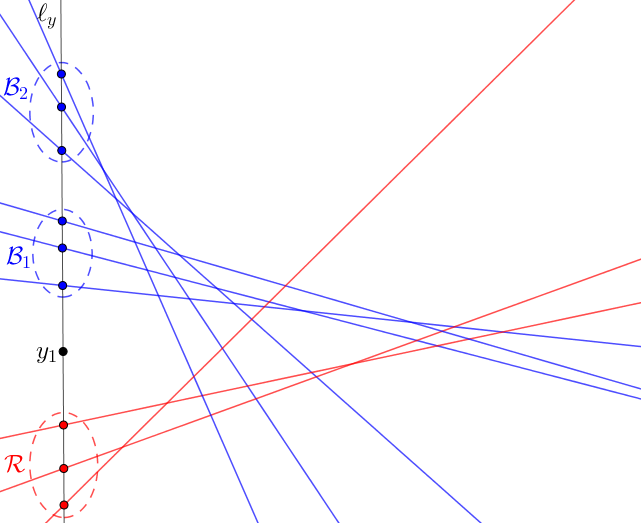}
\caption{ Sets $\mathcal{R}, \mathcal{B}_1, \mathcal{B}_2$ in the proof of Lemma \ref{main lemma 2}.}
\label{First Step}
\end{figure}
\noindent Given a set of $n$ points $P$ and a set of $n$ lines $\mathcal{L}$ in the plane, if $|I(P, \mathcal{L})| \geq cn^{\frac{4}{3}-\frac{1}{9k-6}}$, where $c$ is a sufficiently large constant depending on $k,$ then Corollary \ref{first corollary} implies that there are two sets of $k$ lines such that each pair of them from different sets intersects at a unique point in $P.$ Therefore, Theorem \ref{Second Corollary} follows  by combining  Theorem \ref{main theorem} with the following lemma.


\begin{Lem} \label{main lemma 2} \label{natural grid}
There is a natural number $c$ such that the following holds. Let $\mathcal{B}$ be a set of $ct^2$ blue lines in the plane, and let $ \mathcal{R}$ be a set of $ct^2$ red lines in the plane such that for $P=\{\ell_1 \cap \ell_2: \ell_1 \in \mathcal{B}, \ell_2 \in \mathcal{R}\}$ we have $|P|=c^2t^4.$ Then $(P, \mathcal{B} \cup \mathcal{R})$ contains a natural $t \times t$ grid. 
\end{Lem}

To prove Lemma \ref{main lemma 2}, we will need the following lemma which is an immediate consequence of Dilworth's Theorem.

\begin{Lem} \label{Dilworth}
For $n >0,$ let $\mathcal{L}$ be a set of $n^2$ lines in the plane, such that no two members intersect the same point on the $y$-axis. Then there is a subset $\mathcal{L}' \subset \mathcal{L}$ of size $n$ such that the intersection point of any two members in $\mathcal{L}'$ lies to the left of the $y$-axis, or the intersection point of any two members in $\mathcal{L}'$ lies to the right of the $y$-axis. 
\end{Lem}

\begin{proof}
Let us order the elements in $\mathcal{L}=\{\ell_1, \ldots, \ell_{n^2}\}$ from bottom to top according to their $y$-intercept. By Dilworth's Theorem \cite{dilworth1950decomposition}, $\mathcal{L}$ contains a subsequence of $n$ lines whose slopes are either increasing or decreasing. In the first case, all intersection points are to the left of the $y$-axis, and in the latter case, all intersection points are to the right of the $y$-axis.
\end{proof}

\begin{proof}[Proof of Lemma \ref{main lemma 2}.]
Let $(P, \mathcal{B} \cup \mathcal{R})$ be as described above, and let $\ell_y$ be the $y$-axis. Without loss of generality, we can assume that all lines in $\mathcal{B} \cup \mathcal{R}$ are not vertical, and the intersection point of any two lines in $\mathcal{B} \cup \mathcal{R}$ lies to the right of $\ell_y.$ Moreover, we can assume that no two lines intersect at the same point on $\ell_y.$

We start by finding a point $y_1 \in \ell_y$ such that at least $|\mathcal{B}|/2$ blue lines in $\mathcal{B}$ intersect $\ell_y$ on one side of the point $y_1$ (along $\ell_y$) and at least $|\mathcal{R}|/2$ red lines in $\mathcal{R}$ intersect $\ell_y$ on the other side. This can be done by sweeping the point $y_1$ along $\ell_y$ from bottom to top until $ct^2/2$ lines of the first color, say red, intersect $\ell_y$ below $y_1.$ We then have at least $ct^2/2$ blue lines intersecting $\ell_y$ above $y_1.$ Discard all red lines in $\mathcal{R}$ that intersect $\ell_y$ above $y_1,$ and discard all blue lines in $\mathcal{B}$ that intersect $\ell_y$ below $y_1.$ Hence, $|\mathcal{B}| \geq ct^2/2.$ 

\begin{figure}
\centering
\includegraphics[width=0.6\textwidth]{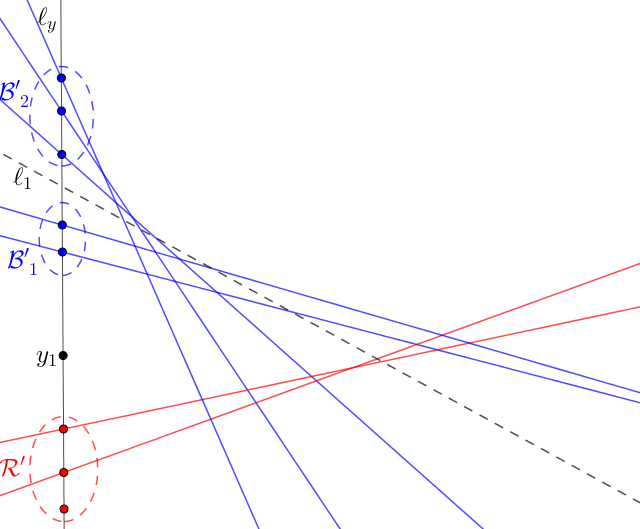}
\caption{An example for the line $\ell_1.$}
\label{Second Step}
\end{figure}

Set $s= \lfloor ct^2/4 \rfloor.$ For the remaining lines in $\mathcal{B},$ let $\mathcal{B}=\{b_1, \ldots, b_{2s}\},$ where the elements of $\mathcal{B}$ are ordered in the order they cross $\ell_y,$ from bottom to top.  We partition $\mathcal{B}= \mathcal{B}_1 \cup \mathcal{B}_2$ into two parts, where $\mathcal{B}_1=\{b_1, \ldots, b_{s}\}$ and $ \mathcal{B}_2=\{b_{s+1}, \ldots, b_{2s}\}.$ By applying an affine transformation, we can assume all lines in $\mathcal{R}$ have positive slope and all lines in $\mathcal{B}_1 \cup \mathcal{B}_2$ have negative slope. See Figure \ref{First Step}.

Let us define a $3$-partite $3$-uniform hypergraph $H=(\mathcal{R} \cup  \mathcal{B}_1 \cup \mathcal{B}_2, E),$ whose vertex parts are $\mathcal{R},\mathcal{B}_1,\mathcal{B}_2,$   and $(r,b_i,b_j) \in \mathcal{R} \times \mathcal{B}_1 \times \mathcal{B}_2$ is an edge in $H$ if and only if the intersection point $p=b_i\cap b_j$ lies above the line $r.$ Note, if $b_i$ and $b_j$ are parallel, then $(r,b_i,b_j) \notin E.$ Then a result of Fox et al.~on semi-algebraic hypergraphs implies the following (see also \cite{bukh2012space} and \cite{fox2016polynomial}). 

\begin{Lem}[Fox et al.~\cite{fox2012overlap}, Theorem 8.1]\label{Semialgebraic}
There exists a positive constant $\alpha$ such that the following holds. In the hypergraph above, there are subsets $\mathcal{R}' \subseteq \mathcal{R}, \mathcal{B}'_1 \subseteq \mathcal{B}_1, \mathcal{B}'_2 \subseteq \mathcal{B}_2,$ where $|\mathcal{R}'| \geq \alpha |\mathcal{R}|, |\mathcal{B}'_1| \geq \alpha |\mathcal{B}_1|,|\mathcal{B}'_2|\geq \alpha |\mathcal{B}_2|,$ such that either $\mathcal{R}' \times \mathcal{B}'_1 \times \mathcal{B}'_2 \subseteq E,$ or $(\mathcal{R}' \times \mathcal{B}'_1 \times \mathcal{B}'_2) \cap E= \emptyset.$ 
\end{Lem}

We apply Lemma \ref{Semialgebraic} to $H$ and obtain subsets $\mathcal{R}', \mathcal{B}'_1, \mathcal{B}'_2$ with the properties described above. Without loss of generality, we can assume that $\mathcal{R}' \times \mathcal{B}'_1 \times \mathcal{B}'_2 \subset E,$ since a symmetric argument would follow otherwise. Let $\ell_1$ be a line in the plane such that the following holds.

\begin{enumerate}
\item The slope of $\ell_1$ is negative.

\item All intersection points between $\mathcal{R'}$ and $\mathcal{B'}_1$ lie above $\ell_1.$

\item All intersection points between $\mathcal{R'}$ and $\mathcal{B'}_2$ lie below $\ell_1.$
\end{enumerate}
See Figure \ref{Second Step}.

\begin{Obs}\label{Existence of l_1}
Line $\ell_{1}$ defined above exists.
\end{Obs}

\begin{proof}
Let $U$ be the upper envelope of the arrangement $\bigcup_{\ell \in \mathcal{R}'} \ell,$ that is, $U$ is the closure of all points that lie on exactly one line of $\mathcal{R}'$ and strictly above exactly the $|\mathcal{R}'|-1$ lines in $\mathcal{R}'.$ 

Let $P_1$ be the set of intersection points between the lines in $\mathcal{B}'_1$ with $U.$ Likewise, we define $P_2$ to be the set of intersection points between the lines in $\mathcal{B}'_2$ with $U.$ Since $U$ is $x$-monotone and convex the set $P_2$ lies to the left of the set $P_1.$
Then the line $\ell_1$ that intersects $U$ between $P_1$ and $P_2$ and intersects $\ell_y$ between $\mathcal{B}'_1$ and $\mathcal{B}'_2$ satisfies the conditions above.\end{proof}

Now we apply Lemma \ref{Dilworth} to $\mathcal{R}'$ with respect to the line $\ell_1,$ to obtain $\sqrt{\alpha c/2} \cdot t$ members in $\mathcal{R}'$ such that every pair of them intersects on one side of $\ell_1.$ Discard all other members in $\mathcal{R}'.$ Without loss of generality, we can assume that all intersection points between any two members in $\mathcal{R}'$  lie below $\ell_1,$  since a symmetric argument would follow otherwise. We now discard the set $\mathcal{B}'_2.$ 

Notice that the order in which the lines in $\mathcal{R}'$ cross $b \in \mathcal{B}'_1$ will be the same for any line $b \in \mathcal{B}'_1.$ Therefore, we order the elements in $\mathcal{R}'=\{r_1, \ldots, r_m\}$ with respect to this ordering, from left to right, where $m=\lceil \sqrt{\alpha c/2} \cdot t \rceil.$ We define $\ell_2$ to be the line obtained by slightly perturbing the line $r_{\lfloor m/2 \rfloor}$ such that:

\begin{enumerate}
\item The slope of $\ell_2$ is positive.

\item  All intersection points between $\mathcal{B}'_1$ and $\{r_1, \ldots, r_{\lfloor m/2 \rfloor}\}$ lie above $\ell_2.$

\item All intersection points between $\mathcal{B}'_1$ and $\{r_{\lfloor m/2 \rfloor+1}, \ldots, r_{m}\}$ lie below $\ell_2.$ 
\end{enumerate}

\noindent See the Figure \ref{Fourth Step}. 

\begin{figure}
\centering
\includegraphics[width=0.7\textwidth]{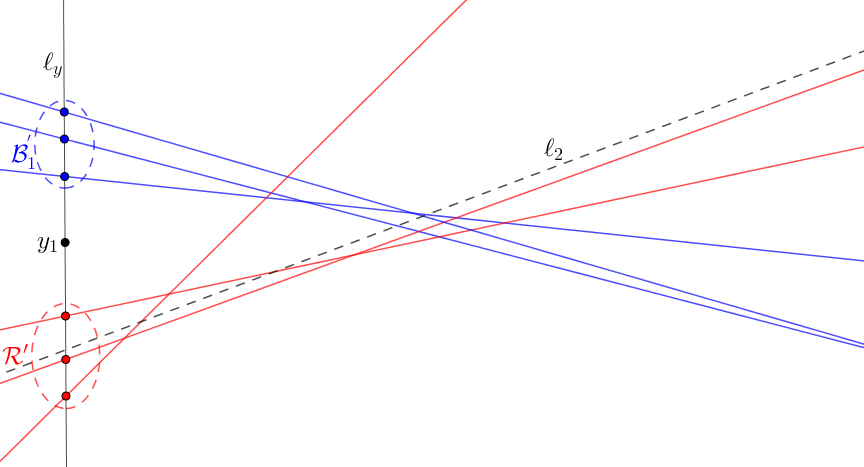}
\caption{An example for the line $\ell_2.$}
\label{Fourth Step}
\end{figure}

Finally, we apply Lemma \ref{Dilworth} to $\mathcal{B}'_1$ with respect to the line $\ell_2,$ to obtain at least $\sqrt{\alpha c} \cdot t/2$ members in $\mathcal{B}'_1$ with the property that any two of them intersect on one side of $\ell_2.$ Without loss of generality, we can assume that any two such lines intersect below  $\ell_2$ since a symmetric argument would follow. Set $\mathcal{B}^{\ast} \subset \mathcal{B}'_1$ to be these set of lines. Then $\mathcal{B}^{\ast} \cup \{r_{1}, \ldots, r_{\lfloor m/2 \rfloor}\}$ and their intersection points form a natural grid. By setting $c=c(t)$ to be sufficiently large, we obtain a natural $t \times t$ grid. \end{proof}

\section{Lower Bound Construction} \label{Lower Bound Construction}

In this section, we will prove Theorem \ref{Lower Bound}. First, let us recall the definitions of Sidon and $k$-fold Sidon sets.

Let $A$ be a finite set of positive integers. Then $A$ is a \emph{Sidon set} if the sum of all pairs are distinct, that is, the equation $x+y=u+v$ has no solutions with $x,y,u,v  \in A$, except for trivial solutions given by $u=x, y=v$ and $x=v, y=u.$ We define $s(N)$ to be the size of the largest Sidon set $A \subset \{1, \ldots, N\}.$ Erd\H{o}s and Tur\'an proved the following. 

\begin{Lem}[See \cite{erdos1941problem} and \cite{ruzsa1993solving}]\label{Sidon Set}
For $N>1,$ we have $s(N)=\Theta(\sqrt{N}).$
\end{Lem}

Let us now consider a more general equation. Let $u_1, \ldots, u_4$ be integers such that $u_1+u_2+u_3+u_4=0,$ and consider the equation 
\begin{align}\label{mysterious equation}
u_1x_1+u_2x_2+u_3x_3+u_4x_4=0.
\end{align}

We are interested in solutions to (\ref{mysterious equation}) with $x_1, x_2, x_3, x_4 \in \mathbb{Z}.$ Suppose $(x_1, x_2, x_3, x_4)=(a_1, a_2, a_3, a_4)$ is an integer solution to (\ref{mysterious equation}). Let $d \leq 4$ be the number of distinct integers in the set $\{a_1, a_2, a_3, a_4\}.$ Then we have a partition on the indices
\begin{align*}
\{1,2,3,4\}=T_1 \cup \cdots \cup T_{d},
\end{align*}
\noindent  where $i$ and $j$ lie in the same part $T_{\nu}$  if and only if $x_i=x_j.$ We call $(a_1, a_2, a_3, a_4)$ a \emph{trivial} solution to (\ref{mysterious equation}) if 
\begin{align*}
\sum_{i \in T_{\nu}}u_i=0,      \hspace{2cm}    \nu=1,\ldots, d.
\end{align*}
Otherwise, we will call $(a_1, a_2, a_3, a_4)$ a \emph{nontrivial} solution to (\ref{mysterious equation}).

In \cite{lazebnik2003hypergraphs}, Lazebnik and Verstra\"ete introduced $k$-fold Sidon sets which are defined as follows. Let $k$ be a positive integer. A set $A \subset \mathbb{N}$ is a \emph{$k$-fold Sidon set} if each equation of the form 
\begin{align}\label{mysterious equation2}
u_1x_1+u_2x_2+u_3x_3+u_4x_4=0,
\end{align}
where $|u_i| \leq k$ and $u_1+\cdots+u_4=0,$  has no nontrivial solutions with $x_1, x_2, x_3, x_4 \in A.$ Let $r(k,N)$ be the size of the largest $k$-fold Sidon set $A \subset \{1, \ldots, N\}.$

\begin{Lem}\label{Ruzsa's Lemma}
There is an infinite sequence $1=a_1 < a_2 < \cdots$ of integers such that \[a_m \leq 2^8k^4m^3,\] and the system of equations (\ref{mysterious equation2}) has no nontrivial solutions in the set $A=\{a_1, a_2, \ldots\}.$
In particular, for integers $N > k^4 \geq 1,$ we have $r(k,N) \geq
ck^{-{4}/{3}}{N^{{1}/{3}}},$ where $c$ is a positive constant.
\end{Lem}

The proof of Lemma \ref{Ruzsa's Lemma} is a slight modification of the proof of Theorem $2.1$ in \cite{ruzsa1993solving}. For the sake of completeness, we include the proof here.
\begin{proof}
We put $a_1=1$ and define $a_m$ recursively. Given $a_1, \ldots, a_{m-1},$ let $a_m$ be the smallest positive integer satisfying 
\begin{align}\label{recursive formula}
a_m \neq - \Big( \sum_{i \in S} u_i \Big)^{-1} \sum_{1 \leq i \leq 4, i \notin S} u_ix_i,
\end{align}
for every choice $u_i$ such that $|u_i| \leq k,$ for every set $S \subset \{1, \ldots,4\}$ of subscripts such that $\Big( \sum_{i \in S} u_i \Big) \neq 0$, and for every choice of $x_i \in \{a_1, \ldots, a_{m-1}\}$, where $i \notin S.$ For a fixed $S$ with $|S|=j,$ this excludes $(m-1)^{4-j}$ numbers. Since $|u_i|\leq k,$ the total number of excluded integers is at most 
\[({2k+1})^{4}\sum_{j=1}^{3}\binom{4}{j}(m-1)^{4-j}=({2k+1})^{4}(m^4-(m-1)^4-1) < 2^8k^4m^3.\]

Consequently, we can extend our set by an integer $a_m \leq 2^8k^4m^3.$ This will automatically be different from from $a_1, \ldots, a_{m-1},$ since putting $x_i=a_j$ for all $i \notin S$ in (\ref{recursive formula}) we get $a_m \neq a_j.$ It will also satisfy $a_m > a_{m-1}$ by minimal choice of $a_{m-1}.$

We show that the system of equations (\ref{mysterious equation2}) has no nontrivial solutions in the set $\{a_1, \ldots, a_m\}.$ We use induction on $m.$ The statement is obviously true for $m=1.$ We establish it for $m$ assuming for $m-1.$ Suppose that there is a nontrivial solution $(x_1,x_2,x_3,x_4)$ to (\ref{mysterious equation2}) for some $u_1, u_2, u_3, u_4$ with the properties described above. Let $S$ denote the set of those subscripts for which $x_i=a_m.$ If $\sum_{i \in S}u_i \neq 0,$ then this contradicts (\ref{recursive formula}). If $\sum_{i \in S}u_i = 0,$ then by replacing each occurrence of $a_m$ by $a_1,$ we get another nontrivial solution, which contradicts the induction hypothesis.
\end{proof}

For more problems and results on Sidon sets and $k$-fold Sidon sets, we refer the interested reader to \cite{lazebnik2003hypergraphs, ruzsa1993solving, cilleruelo2014k}.

We are now ready to prove Theorem \ref{Lower Bound}.

\begin{proof}[Proof of Theorem \ref{Lower Bound}.] We start by applying Lemma \ref{Sidon Set} to obtain a Sidon set $M \subset [n^{1/{7}}],$ such that $|M|=\Theta(n^{1/{14}}).$ We then apply Lemma \ref{Ruzsa's Lemma} with $k=n^{1/7}$ and $N=\frac{1}{4}n^{{11}/{14}},$ to obtain a $k$-fold Sidon set $A \subset [N]$ such that \[|A| \geq cn^{1/14} ,\] where $c$ is defined in Lemma \ref{Ruzsa's Lemma}. Without loss of generality, let us assume $|A|=cn^{1/14}.$

Let $P= \{(i,j)\in \mathbb{Z}^2: i \in A, 1 \leq j \leq n^{13/14}\},$ and let  $\mathcal{L}$ be the family of lines in the plane of the form $y=mx+b, $ where $m \in M$ and  $b$ is an integer such that $1 \leq b \leq n^{13/14}/2.$ 

\noindent Hence, we have
\begin{align*}
|P|&=|A| \cdot n^{13/14}=\Theta(n),\\
|\mathcal{L}|&=|M| \cdot \frac{n^{13/14}}{2}=\Theta(n).
\end{align*}

\noindent Notice that each line in $\mathcal{L}$ has exactly $|A|=cn^{1/14}$ points from $P$ since $1 \leq b \leq n^{13/14}/2.$ Therefore,
\begin{align*}
|I(P,\mathcal{L})|&=|\mathcal{L}||A|=\Theta(n^{1+1/14}).
\end{align*}

\begin{Claim}
There are no four distinct lines ${\ell}_1, {\ell}_2,{\ell}_3,{\ell}_4 \in \mathcal{L}$ and four distinct points $p_1, p_2, p_3,p_4 \in P$ such that ${\ell}_1 \cap {\ell}_2 =p_1, {\ell}_2 \cap {\ell}_3=p_2, {\ell}_3 \cap {\ell}_4=p_3, {\ell}_4 \cap {\ell}_1= p_4.$ 
\end{Claim}

\begin{proof}
For the sake of contradiction, suppose there are four lines $\ell_1, \ell_2, \ell_3, \ell_4$ and four points $p_1, p_2,p_3, p_4$ with the properties described above. Let $\ell_i=m_ix+b_i$ and let $p_i=(x_i,y_i).$
Therefore, 
\begin{equation*}
{\ell}_1 \cap {\ell}_2=p_1=(x_1, y_1),
\end{equation*}
\begin{equation*}
{\ell}_2 \cap {\ell}_3=p_2=(x_2, y_2),
\end{equation*}
\begin{equation*}
{\ell}_3 \cap {\ell}_4=p_3=(x_3, y_3),
\end{equation*}
 \begin{equation*}
{\ell}_4 \cap {\ell}_1=p_4=(x_4, y_4).
 \end{equation*}

\noindent Hence,

\begin{align*}
p_1 &\in {\ell}_1, {\ell}_2  \hspace{3mm} \Longrightarrow (m_1-m_2)x_1+b_1-b_2=0,
\end{align*}
\begin{align*}
 p_2 &\in {\ell}_2, {\ell}_3  \hspace{3mm}  \Longrightarrow (m_2-m_3)x_2+b_2-b_3=0,
\end{align*}
\begin{align*}
p_3 &\in {\ell}_3, {\ell}_4 \hspace{3mm}   \Longrightarrow  (m_3-m_{4})x_3+b_3-b_{4}=0,
\end{align*}
\begin{align*}
p_4 &\in {\ell}_4, {\ell}_1 \hspace{3mm}  \Longrightarrow (m_{4}-m_1)x_4+b_{4}-b_1=0.
\end{align*}

\noindent By summing up the four equations above, we get 
\begin{align*}
(m_1-m_2)x_1+(m_2-m_3)x_2+(m_3-m_4)x_3+(m_4-m_1)x_4=0.
\end{align*} 

\noindent By setting $u_1=m_1-m_2, u_2=m_2-m_3, u_3=m_3-m_4,  u_4=m_4-m_1,$ we get

\begin{align}\label{final}
u_1x_1+u_1x_2+u_3x_3+u_4x_4=0,
\end{align}

\noindent  where $u_1+u_2+u_3+u_4=0$ and $|u_i| \leq n^{1/{7}}.$ Since $x_1, \ldots, x_4 \in A,$ $(x_1, x_2,x_3,x_4)$ must be a trivial solution to (\ref{final}). The proof now falls into the following cases, and let us note that no line in $\mathcal{L}$ is vertical.

\medskip

\noindent \emph{Case 1}. Suppose $x_1=x_2=x_3=x_4.$ Then $\ell_i$ is vertical and we have a contradiction.

\medskip

\noindent \emph{Case 2}. Suppose $x_1=x_2=x_3 \neq x_4$ and $u_1+u_2+u_3=0$ and $u_4=0.$ Then $\ell_1$ and $\ell_4$ have the same slope which is a contradiction. The same argument follows if $x_1=x_2=x_4 \neq x_3,$ $
x_1=x_3=x_4 \neq x_2,$ or $x_2=x_3=x_4 \neq x_1.$

\medskip

\noindent \emph{Case 3}. Suppose $x_1=x_2 \neq x_3=x_4,$ $u_1+u_2=0,$ and $u_3+u_4=0.$ Since $
p_1,p_2 \in \ell_2$ and $x_1=x_2,$ this implies that $\ell_2$ is vertical which is a contradiction.  A similar argument follows if $x_1=x_4 \neq x_2=x_3,$ $u_1+u_4=0,$ and $u_2+u_3=0.$
\medskip

\noindent \emph{Case 4}. Suppose $x_1=x_3 \neq x_2=x_4,$ $u_1+u_3=0,$ and $u_2+u_4=0.$ Then $u_1+u_3=0$ implies that $m_1+m_3=m_2+m_4.$ Since $M$ is a Sidon set, we have either $m_1=m_2$ and $m_3=m_4$ or $m_1=m_4$ and $m_2=m_3.$  The first case implies that $\ell_1$ and $\ell_2$ are parallel which is a contradiction, and the second case implies that $\ell_2$ and $\ell_3$ are parallel, which is again a contradiction. \end{proof}  

This completes the proof of Theorem \ref{Lower Bound}.\end{proof} 

\section{Concluding Remarks}
\begin{itemize}

\item An old result of Erd\H{o}s states that every $n$-vertex graph that does not contain a cycle of length $2k,$ has $O_k(n^{1+1/k})$ edges. It is known that this bound is tight when $k=2, 3,$ and $5$, but it is a long standing open problem in extremal graph theory to decide whether or not this upper bound can be improved for other values of $k.$ Hence, Erd\H{o}s's upper bound of $O(n^{5/4})$ when $k=4$ implies Theorem \ref{main theorem} when $t=2$ and $m=n.$ It would be interesting to see if one can improve the upper bound in Theorem \ref{main theorem} when $t=2.$ For more problems on cycles in graphs, see \cite{verstraete2016extremal}.

\item The proof of Lemma \ref{main lemma 2} is similar to the proof of the main result in \cite{aronov1994crossing}. The main difference is that we use the result of Fox et al.~\cite{fox2012overlap} instead of the Ham-Sandwich Theorem. We also note that a similar result was established by Dujmovi\'c and Langerman (see Theorem $6$ in \cite{Dujmovic2013}).

\item Recently, Tomon and the second author \cite{ST20} improved the lower bound in Theorem \ref{Lower Bound} to $n^{9/8  + o(1)}$, and more generally, gave a construction of $n$ points and $n$ lines in the plane with no $k\times k$ grid and with at least $n^{4/3 - \Theta(1/k)}$ incidences.

\end{itemize}

\vspace{1mm}
\bibliographystyle{acm}
\bibliography{bibfile}
\vspace{1mm}

\end{document}